\documentclass[12pt,reqno]{amsart}

\usepackage[T1]{fontenc}

\usepackage{mathtools}
\usepackage{amsmath}								
\usepackage{amssymb}
\usepackage{amsthm}
\usepackage{amscd}
\usepackage{amsfonts}
\usepackage{stmaryrd}
\usepackage{algorithm, algorithmic}
\usepackage{ wasysym }
\usepackage{listings}

\usepackage{fullpage}
\usepackage[shortlabels]{enumitem}

\usepackage[colorlinks, linktocpage, citecolor = blue, linkcolor = blue]{hyperref}
\usepackage{color}

\DeclareMathOperator{\Cay}{Cay}

\newtheorem{theorem}{Theorem}[section]

\newtheorem{proposition}[theorem]{Proposition}

\theoremstyle{definition}

\newtheorem{example}[theorem]{Example}

\makeatletter
\@namedef{subjclassname@2020}{%
  \textup{2020} Mathematics Subject Classification}
\makeatother

\title{Real symmetric matrices with partitioned eigenvalues}
\author{Madeleine Weinstein}


\begin{document}

\begin{abstract}
We study the real algebraic variety of real symmetric matrices with eigenvalue multiplicities determined by a partition. We present formulas for the dimension and Euclidean distance degree. We give a parametrization by rational functions. For small matrices, we provide equations; for larger matrices, we explain how to use representation theory to find equations. We describe the ring of invariants under the action of the orthogonal group. For the subvariety of diagonal matrices, we give the degree. 
\end{abstract}

\maketitle

\section{Introduction}\label{sec:introduction}

Let $\lambda=(\lambda_1, \dots, \lambda_m)$ be a partition of $n$. Let $\mathbb{R}^{\frac{n(n+1)}{2}}$ be the space of real symmetric $n \times n$ matrices. We define the \textbf{variety of $\lambda$-partitioned eigenvalues} $V_{\mathbb{R}}(\lambda) \subset \mathbb{R}^{\frac{n(n+1)}{2}}$ to be the Zariski closure of the locus of matrices with eigenvalue multiplicities determined by $\lambda$. Since we take the Zariski closure, these varieties include all matrices with eigenvalue multiplicities determined by partitions of $n$ that are coarser than $\lambda$. 

The space of real symmetric matrices has multiple advantages over the spaces of real square matrices, complex square matrices, and complex symmetric matrices. Unlike other real matrices, real symmetric matrices have all real eigenvalues. Additionally, the real symmetric case has better properties with respect to diagonalizability than complex square or complex symmetric matrices.

We illustrate these properties with the example of $n=2,\lambda=(2)$, the locus of $2 \times 2$ matrices with coinciding eigenvalues. 
\begin{example}
Complex $2 \times 2$ matrices with the repeated eigenvalue $\mu$ can have two Jordan normal forms. The first is diagonal and the second is not. For convenience, call a $2 \times 2$ matrix with coinciding eigenvalues type A if its Jordan normal form (JNF) is diagonal and type B otherwise: 
\[ 
\text{JNF of a Type A matrix}: \begin{pmatrix}
\mu & 0 \\
0 &  \mu
\end{pmatrix}
\hspace{1cm}
\text{JNF of a Type B matrix}: \begin{pmatrix}
\mu & 1 \\
0 &  \mu
\end{pmatrix}. 
 \] 

We examine the dimensions of the loci of type A and type B matrices in three cases: complex square, complex symmetric, and real symmetric. In each case, the dimension of the locus of type A matrices is $1$ because scalar matrices are fixed by conjugation. For complex square matrices, the dimension of the Type B locus is 3. For complex symmetric matrices, the dimension of the Type B locus is 2. Conjugating the JNF of a Type B matrix by any invertible matrix of the form 
\[ \begin{pmatrix}
a & b \\
c & d
\end{pmatrix} \]
where $a^2+c^2=0$ yields a symmetric matrix with coinciding eigenvalues. Real symmetric matrices are orthogonally diagonalizable, so the Type B locus for real symmetric matrices is empty. 

The significance of these dimensions is as follows. A generic real symmetric matrix with coinciding eigenvalues is diagonalizable and a generic complex square matrix or complex symmetric matrix with coinciding eigenvalues is not. Real symmetric matrices can be studied through their diagonalizations. 
\end{example}

Matrices with repeated eigenvalues have been studied in contexts from geometry \cite{BKL} to classical invariant theory and linear algebra \cite{ Domokos1,Domokos3,Domokos2,Ily,Lax, Parlett,RoySOS, BerndPoly}. Recently, they have come to focus in the study of curvature of algebraic varieties \cite{BRW}. The \textit{principal curvatures} of a variety are the eigenvalues of the \textit{second fundamental form}. Coincidences of eigenvalues correspond to geometric features; for example, on a surface, a point where the eigenvalues of the second fundamental form coincide is called an \textit{umbilic}. At an umbilic the best second-order approximation of a surface is given  by a sphere. 

A matrix is called \textit{degenerate} if its eigenvalues are not all distinct. The locus of such matrices is a variety defined by the matrix discriminant. As a polynomial in the eigenvalues, the matrix discriminant is the product of the squared differences of each pair of eigenvalues, and thus it is zero exactly when the eigenvalues are not distinct. To study the variety of degenerate matrices, one considers the discriminant as a polynomial in the entries of the matrix. In \cite{Parlett}, Parlett gives an equation for the discriminant of a matrix in its entries by describing it as the determinant of another matrix. 

In a refinement of the study of degenerate matrices, some authors \cite{Domokos3, Domokos2, RoySOS} have studied matrices by their number of distinct eigenvalues. In this situation, the role of the matrix discriminant is played by the sequence of $k$-\textit{subdiscriminants}. The $0$-subdiscriminant is the usual matrix discriminant. An $n \times n$ matrix has exactly $n-k$ distinct eigenvalues if and only if subdiscriminants $0$ through $k-1$ vanish and the $k$-subdiscriminant does not. In \cite{RoySOS}, Roy gives an explicit description of the $k$-subdiscriminant of the characteristic polynomial of a matrix $A$ in terms of the entries of $A$. Furthermore, she expresses the $k$-subdiscriminant as a sum of squares with real coefficients. 

In \cite{Domokos2}, Domokos studies the variety of matrices with a bounded number of distinct eigenvalues from the perspective of invariant theory. This variety can be characterized by its invariance under the action of conjugation by the special orthogonal group on the space of symmetric matrices. He describes the minimal degree homogeneous component of the space of invariants of the variety of matrices with a bounded number of distinct eigenvalues.

The variety of $\lambda$-partitioned eigenvalues appears in \cite{EDDTransfer}, where Bik and Draisma analyze its properties with respect to distance optimization. 
In \cite{Khaz}, Kozhasov studies the open submanifold of the variety of $\lambda$-partitioned eigenvalues where the eigenvalues have exact multiplicities $\lambda$. Kozhasov proves that it is a \textit{minimal submanifold} of the space of real symmetric $n \times n$ matrices. A minimal submanifold of a Riemannian manifold is one with zero mean curvature vector field; this generalizes the concept of surfaces in $\mathbb{R}^3$ that locally minimize area. 

This paper further investigates the variety of $\lambda$-partitioned eigenvalues, with each section addressing a different aspect. In Section \ref{sec:parametrization}, we give a parametrization.  In Section \ref{sec:dimension}, we prove a formula for the dimension.  In Section \ref{sec:equations}, we put the parametrization to work to find equations and the degree of this variety for small $n$. We explain how representation theory can be used to extend these calculations to larger $n$. We also describe the ring of $O(n)$-invariants. In Section \ref{sec:diagonal}, we study the diagonal restriction, computing its degree. We show how the diagonal restriction provides a good model for distance optimization questions regarding the variety of $\lambda$-partitioned eigenvalues, presenting a theorem of Bik and Draisma for its Euclidean distance degree.

 \section{Parametrization}\label{sec:parametrization}

In this section, we describe a parametrization of the
variety of $\lambda$-partitioned eigenvalues $V_{\mathbb{R}}(\lambda) \subset \mathbb{R}^{\frac{n(n+1)}{2}}$ by rational functions. Real symmetric matrices are diagonalizable by orthogonal matrices. The special orthogonal group is parametrized by the set of skew-symmetric matrices.

\begin{proposition}\label{prop:parametrization} Let $\lambda=(\lambda_1, \dots, \lambda_m)$ be a partition of $n$ such that $\lambda \neq (1, \dots, 1)$. Let $Diag(\lambda)$ be a diagonal $n \times n$ matrix with diagonal entries $\mu_1, \dots, \mu_m$ where each entry $\mu_i$ appears with multiplicity $\lambda_i$. Let $B$ be a skew-symmetric $n \times n$ matrix. Let $I$ be the $n \times n$ identity matrix. The map 
\[ p: \mathbb{R}^{n \times n} \to \mathbb{R}^{n \times n} \] 
\[ B \mapsto (I-B)(I+B)^{-1}Diag(\lambda)(I+B)(I-B)^{-1} \] 
is a parametrization of a Zariski open dense subset of $V_{\mathbb{R}}(\lambda)$ by rational functions. 
\end{proposition}

\begin{proof}
Consider the \text{Cayley transform} map from the space $Skew(n)$ of real skew-symmetric $n \times n$ matrices to the orthogonal group $O(n)$ of real $n \times n$ matrices:
\[ \Cay: Skew(n) \to O(n) \]
\[   B \mapsto (I-B)(I+B)^{-1}  \] 
where $I$ is the $n \times n$ identity matrix. Its  image is the set $SO(n)$ of special orthogonal matrices minus those with $-1$ as an eigenvalue. See \cite{AC} for details. It is well known that real symmetric matrices are diagonalizable by orthogonal matrices. We show that real symmetric matrices with a repeated eigenvalue are diagonalizable by special orthogonal matrices. Let $A=PDP^{T}$ where $P$ is an orthogonal matrix, $D$ is a diagonal matrix, and $A$ has a repeated eigenvalue. Then we can choose $i \neq j$ such that $D_{(i,i)}=D_{(j,j)}$. Let $P_{(ij)}$ denote the matrix obtained by swapping row $i$ and row $j$ of $P$. If $P$ has determinant $-1$, then $P_{(ij)}$ has determinant $1$, so $A=P_{(ij)}DP^T_{(ij)}$ is a diagonalization of $A$ by a special orthogonal matrix. The proposition follows.
\end{proof}

This parametrization will be used in the next section to find a formula for the dimension of $V_{\mathbb{R}}(\lambda)$ and in the numerical computations of Section \ref{sec:equations}.

\section{Dimension}\label{sec:dimension}

The dimension is an important invariant of any algebraic variety. We now give a formula for the dimension of the variety of $\lambda$-partitioned eigenvalues and show how it can be proved using the fact that this variety is the $O(n)$-orbit of a certain form of diagonal matrix. To use the algebraic definition of dimension, we study the complexification of $V_{\mathbb{R}}(\lambda)$, denoted $V_{\mathbb{C}}(\lambda)$, which is characterized by its parametrization. 

 \begin{theorem}\label{thm:dimension}
 The complexification $V_{\mathbb{C}}(\lambda)$ of the real algebraic variety $V_{\mathbb{R}}(\lambda) \subset \mathbb{R}^{\frac{n(n+1)}{2}}$ of $n \times n$ real symmetric matrices with eigenvalue multiplicities corresponding to the partition $\lambda= (\lambda_1, \dots, \lambda_m)$ of $n$ or partitions coarser than $\lambda$ is an irreducible variety of dimension $m+ \binom{n}{2}-\sum_{i=1}^m \binom{\lambda_i}{2}$.
 \end{theorem}

 \begin{proof} Let $\lambda=(1, \dots, 1)$. Then $V_{\mathbb{C}}(\lambda)=\mathbb{C}^\frac{n(n+1)}{2}$, so the theorem holds. Now suppose $\lambda \neq (1, \dots, 1)$.
 Every real symmetric matrix $A$ can be written in the form $A=PDP^{-1}$ where $P$ is a real orthogonal matrix and $D$ is a diagonal matrix with diagonal entries equal to the eigenvalues of $A$. For the remainder of the proof, we focus on the complexification  $V_{\mathbb{C}}(\lambda)$ of
 $V_{\mathbb{R}}(\lambda)$, which consists of matrices of the form $A=PDP^{-1}$ where $P$ is an orthogonal matrix with complex entries and $D$ is a diagonal matrix with entries partitioned by $\lambda$ or a partition coarser than $\lambda$. As argued in the proof of Proposition \ref{prop:parametrization}, when $A$ has a repeated eigenvalue, it in fact suffices to let $P$ be a special orthogonal matrix. Thus this parametrization gives $V_{\mathbb{C}}(\lambda)$ as the image of the irreducible variety of special orthogonal matrices with complex entries under a regular map, so  $V_{\mathbb{C}}(\lambda)$  is irreducible. This parametrization also enables us to prove a formula for the dimension of $V_{\mathbb{C}}(\lambda)$. 
Suppose there are $m$ distinct eigenvalues. 
 Arrange the matrix $D$ so that repeated eigenvalues are grouped together along the diagonal. Scalar matrices commute with all matrices, so the scalar matrix block corresponding to each eigenvalue commutes with the corresponding blocks of $P$ and $P^{-1}$. Thus any matrix $P$ with orthogonal blocks for each eigenvalue stabilizes $D$. The dimension of the orthogonal group $O(n)$ is $\binom{n}{2}$. So the dimension of the block orthogonal stabilizer of $D$ is $\sum_{i=1}^m \binom{\lambda_i}{2}$. By Proposition 21.4.3 of \cite{TauvelYu}, the dimension of the orbit of a fixed diagonal matrix is $\binom{n}{2}-\sum_{i=1}^m \binom{\lambda_i}{2}$. Since there are $m$ choices of eigenvalues, the dimension of the set of matrices with multiplicities corresponding to $\lambda$ is $m+\binom{n}{2}-\sum_{i=1}^m \binom{\lambda_i}{2}$. The subset of matrices with eigenvalue multiplicities corresponding to partitions coarser than $\lambda$ is of smaller dimension because there are fewer choices of eigenvalues. Thus the dimension of the variety of matrices with eigenvalue multiplicities corresponding to the partition $\lambda$ or partitions coarser than $\lambda$ is as stated. 
 \end{proof} 
 
 Having a formula for the dimension of this variety will help us to find its equations in Section \ref{sec:equations}.
 
\section{Equations}\label{sec:equations}

In this section, we discuss how to find equations for the varieties 
$V_{\mathbb{R}}(\lambda)$. The case $n=2$ was discussed in the introduction. For $n=3$ and most partitions of $n=4$, we use the parametrization from Section \ref{sec:parametrization} to generate points on the variety and then use interpolation to find polynomials that vanish on these points. For larger $n$, the matrices used for interpolation become too large for both our symbolic and numerical methods. We discuss how representation theory can be used to make these computations more feasible. As a first step towards studying the relevant representations, we describe the ring of invariants. 

In Examples \ref{ex1}, \ref{ex2} and \ref{ex3}, we analyze the varieties $V_{\mathbb{R}}(\lambda)$ where $\lambda$ is a partition of $n=3$ or $n=4$. By Theorem \ref{thm:dimension}, we know the codimension of each variety. We generate points and use interpolation to find equations on those points. We use Macaulay2 \cite{M2} to verify that these equations generate a prime ideal of the expected codimension. This confirms that we have indeed found enough equations to generate the desired ideal. 

For $n=4$ and $\lambda=(2,1,1)$, our interpolation code found no polynomials of degree less than or equal to $5$. As this partition is just the case of degenerate $4 \times 4$ matrices, it has been studied by other authors. The ideal is of codimension $2$, generated by the (unsquared) summands in a sum of squares representation of the matrix discriminant. Parlett provides an algorithm using determinants for computing this discriminant and writing it as a sum of many squares \cite{Parlett}. Domokos gives a nonconstructive proof that it can be written as a sum of 7 squares \cite{Domokos1}. 

\begin{example}\label{ex1}
Let $n=3$ and $\lambda=(2,1)$. We confirm the findings of other authors that this ideal is of codimension $2$ and degree $4$ \cite{BerndPoly}. It is generated by the following 7 cubic polynomials. The matrix discriminant is the sum of the squares of these polynomials. 
\begin{tiny} $$ \begin{matrix}
 -x_{11}x_{13}x_{22}+x_{11}x_{13}x_{33}+x_{12}^2x_{13}-x_{12}x_{22}x_{23}+x_{12}x_{23}x_{33}-x_{13}^3+x_{13}x_{22}^2-x_{13}x_{22}x_{33} \\ 
 -x_{12}^2x_{23}+x_{12}x_{13}x_{22}-x_{12}x_{13}x_{33}+x_{13}^2x_{23} \\
-x_{11}x_{13}x_{23}+x_{12}x_{13}^2-x_{12}x_{23}^2+x_{13}x_{22}x_{23}\\
x_{11}x_{12}x_{23}-x_{11}x_{13}x_{22}+x_{11}x_{13}x_{33}-x_{12}x_{22}x_{23}-x_{13}^3+x_{13}x_{22}^2-x_{13}x_{22}x_{33}+x_{13}x_{23}^2\\
-x_{11}x_{12}x_{22}+x_{11}x_{12}x_{33}-x_{11}x_{13}x_{23}+x_{12}^3+x_{12}x_{22}x_{33}-x_{12}x_{23}^2-x_{12}x_{33}^2+x_{13}x_{23}x_{33} \\
-x_{11}^2x_{23}+x_{11}x_{12}x_{13}+x_{11}x_{22}x_{23}+x_{11}x_{23}x_{33}-x_{12}^2x_{23}-x_{12}x_{13}x_{33}-x_{22}x_{23}
x_{33}+x_{23}^3\\
x_{11}^2x_{22}-x_{11}^2x_{33}-x_{11}x_{12}^2+x_{11}x_{13}^2-x_{11}x_{22}^2+x_{11}x_{33}^2 +x_{12}^2x_{22}-x_{13}^2x_{33}   +x_{22}^2x_{33}-x_{22}x_{23}^2-x_{22}x_{33}^2+x_{23}^2x_{33}
\end{matrix}
$$ \end{tiny}
$\!\!\!\!$
\end{example}

\begin{example}\label{ex2}
Let $n=4$ and $\lambda=(3,1)$. The ideal is of codimension $5$ and degree $8$. It is generated by the following 10 quadrics: 

\begin{tiny} $$ \begin{matrix}
-x_{12}x_{34}+x_{13}x_{24} \\
-x_{12}x_{24}+x_{13}x_{34}+x_{14}x_{22}-x_{14}x_{33} \\
-x_{12}x_{34}+x_{14}x_{23} \\
x_{12}x_{33}-x_{12}x_{44}-x_{13}x_{23}+x_{14}x_{24} \\
-x_{12}x_{23}+x_{13}x_{22}-x_{13}x_{44}+x_{14}x_{34} \\
x_{11}x_{34}-x_{13}x_{14}-x_{22}x_{34}+x_{23}x_{24} \\
x_{11}x_{24}-x_{12}x_{14}+x_{23}x_{34}-x_{24}x_{33} \\
-x_{11}x_{33}+x_{11}x_{44}+x_{13}^2-x_{14}^2+x_{22}x_{33}-x_{22}x_{44}-x_{23}^2+x_{24}^2 \\
x_{11}x_{23}-x_{12}x_{13}-x_{23}x_{44}+x_{24}x_{34} \\
-x_{11}x_{22}+x_{11}x_{44}+x_{12}^2-x_{14}^2+x_{22}x_{33}-x_{23}^2-x_{33}x_{44}+x_{34}^2
\end{matrix}
$$ \end{tiny}
$\!\!\!\!$
\end{example}

\begin{example}\label{ex3}
Let $n=4$ and $\lambda=(2,2)$. The ideal is of codimension $4$ and degree $6$. It is generated by the following 9 quadrics: 
\begin{tiny} $$ \begin{matrix}
x_{11}^2+4x_{13}^2-x_{22}^2-4x_{24}^2-2x_{11}x_{33}+x_{33}^2+2x_{22}x_{44}-x_{44}^2
\\
x_{11}x_{12}+x_{12}x_{22}+2x_{13}x_{23}+2x_{14}x_{24}-x_{12}x_{33}-x_{12}x_{44} \\
 x_{11}x_{14}-x_{14}x_{22}+2x_{12}x_{24}-x_{14}x_{33}+2x_{13}x_{34}+x_{14}x_{44}
 \\
 x_{11}x_{13}-x_{13}x_{22}+2x_{12}x_{23}+x_{13}x_{33}+2x_{14}x_{34}-x_{13}x_{44} \\

-x_{11}^2-4x_{14}^2+x_{22}^2+4x_{23}^2-2x_{22}x_{33}+x_{33}^2+2x_{11}x_{44}-x_{44}^2
\\
 2x_{12}x_{14}-x_{11}x_{24}+x_{22}x_{24}-x_{24}x_{33}+2x_{23}x_{34}+x_{24}x_{44} \\
 2x_{12}x_{13}-x_{11}x_{23}+x_{22}x_{23}+x_{23}x_{33}+2x_{24}x_{34}-x_{23}x_{44} \\
  -x_{11}^2-4x_{12}^2+2x_{11}x_{22}-x_{22}^2+x_{33}^2+4x_{34}^2-2x_{33}x_{44}+x_{44}^2 \\
-x_{11}x_{34}+2x_{13}x_{14}-x_{22}x_{34}+2x_{23}x_{24}+x_{33}x_{34}+x_{34}x_{44}
\end{matrix}
$$ \end{tiny}
$\!\!\!\!$

\end{example}

Naive interpolation strategies, both symbolic and numerical, become infeasible as $n$ increases. To reduce the dimensions of the matrices involved and make these linear algebra computations more feasible, we turn to representation theory. See \cite{FH} and \cite[Chapter 10]{MSbook} for details. 

The ideal $I(V_{\mathbb{R}}(\lambda))$ is
stable under the action by conjugation of the real orthogonal group $O(n)$ on the space $\mathbb{R}^{\frac{n(n+1)}{2}}$ of real symmetric $n \times n$-matrices. 
Thus the degree $d$ homogeneous component $I(V_{\mathbb{R}}(\lambda))_d$ is a representation of $O(n)$. 
So to find generators of $I(V_{\mathbb{R}}(\lambda))$, we find representations of $O(n)$. Since $O(n)$ is reductive, every representation has an \textit{isotypic decomposition} into irreducible representations. While these irreducible representations may be of high dimension, they can be studied through their low-dimensional \textit{highest weight spaces}. We refer the reader to \cite[Chapter 26]{FH} for details.  

Here we examine the special case of one-dimensional representations of $O(n)$. These vector spaces contain polynomials that are themselves invariant under the action of $O(n)$, rather than simply generating an ideal which is $O(n)$-stable as an ideal. 

Denote by $I(V_{\mathbb{R}}(\lambda))^{O(n)}$ the graded vector space of $O(n)$-invariant polynomials in $I(V_{\mathbb{R}}(\lambda))$. 
Let $V_{\mathbb{R}}(D_{\lambda})$ denote the intersection of $V_{\mathbb{R}}(\lambda)$ with the variety of diagonal matrices in $\mathbb{R}^{\frac{n(n+1)}{2}}$. We often identify $V_{\mathbb{R}}(D_{\lambda})$ with a variety in $\mathbb{R}^n$. The symmetric group $S_n \subset O(n)$, consisting of the permutation matrices, acts on $V_{\mathbb{R}}(D_{\lambda})$ by permuting the diagonal entries.
Let $I(V_{\mathbb{R}}(D_{\lambda}) )^{S_n}$ be the graded vector space of $S_n$-invariant polynomials in $I(V_{\mathbb{R}}(D_{\lambda}))$. 

\begin{theorem}
$I(V_{\mathbb{R}}(\lambda))^{O(n)}$ and $I(V_{\mathbb{R}}(D_{\lambda}))^{S_n}$ are isomorphic as graded vector spaces. 
\end{theorem}

\begin{proof} Let $\text{Sym}_n$ be the space of real symmetric $n \times n$ matrices and $\text{Diag}_n$ the subspace of diagonal matrices. Consider the degree-preserving linear map $\phi: \mathbb{R}[\text{Sym}_n]^{O(n)} \to \mathbb{R}[\text{Diag}_n]^{S_n}$ given by restriction of functions. Suppose $\phi(f)=\phi(g)$, so that $f$ and $g$ agree for all diagonal matrices. Let $A$ be a real symmetric $n \times n$ matrix. Then $A$ is orthogonally diagonalizable; that is, $A=PDP^{-1}$ for some orthogonal matrix $P$ and diagonal matrix $D$. Since $f$ and $g$ are $O(n)$-invariant, we have $f(A)=f(D)$ and $g(A)=g(D)$. Since $f$ and $g$ agree on diagonal matrices, we have $f(D)=g(D)$, so $\phi$ is injective. We note that $\mathbb{R}[\text{Diag}_n]^{S_n}$ is the ring of symmetric polynomials. It is generated by the coefficients of the characteristic polynomial, which are $O(n)$-invariant polynomials in the matrix entries, and thus in the image of $\phi$. So $\phi$ is surjective and thus an isomorphism of graded vector spaces $\mathbb{R}[\text{Sym}_n]^{O(n)} \cong \mathbb{R}[\text{Diag}_n]^{S_n}$. It induces an isomorphism $I(V_{\mathbb{R}}(\lambda))^{O(n)} \cong I(V_{\mathbb{R}}(D_{\lambda}))^{S_n}$.
\end{proof}

This result is beneficial because $I(V_{\mathbb{R}}(D_{\lambda}))^{S_n}$ is easier to study than $I(V_{\mathbb{R}}(\lambda))^{O(n)}$. We thus turn our study to $I(V_{\mathbb{R}}(D_{\lambda}))$.

\section{Diagonal Matrices}\label{sec:diagonal}

In this section, we study the intersection of the variety of $\lambda$-partitioned eigenvalues with the variety of diagonal matrices. Recall that $V_{\mathbb{R}}(D_{\lambda})$ denotes the intersection of $V_{\mathbb{R}}(\lambda)$ with the variety of diagonal matrices in $\mathbb{R}^{\frac{n(n+1)}{2}}$. We identify this with the variety in $\mathbb{R}^n$ of points with coordinates that have multiplicities given by the partition $\lambda$. 
Then $V_{\mathbb{R}}(D_{\lambda})$ is a union of the  $\frac{n!}{\lambda_1!\cdots \lambda_m!}$ subspaces of dimension $m$ given by permuting the coordinates of the subspace 
\[ V_1= \{ (a_1, \ldots, a_1, a_2, \ldots, a_2, \ldots, a_m, \ldots,  a_m) \mid a_i \in \mathbb{R} \} \] 
where the coordinate $a_i$ is repeated $\lambda_i$ times. Characterizing $V_{\mathbb{R}}(D_{\lambda})$ as a union of linear spaces reveals its degree. 

\begin{proposition} Let $\lambda=(\lambda_1, \dots, \lambda_m)$ be a partition of $n$. 
The degree of the variety $V_{\mathbb{R}}(D_{\lambda})$ of $n \times n$ diagonal matrices with eigenvalue multiplicities partitioned according to $\lambda$ is 
\[ \frac{n!}{\lambda_1! \cdots \lambda_m!} \] 
\end{proposition}

One may ask how well the diagonal restriction of the variety of $\lambda$-partitioned eigenvalues models the variety as a whole. With regards to distance optimization, the diagonal restriction is a quite good model. Let $X \subset \mathbb{R}^n$ be a real algebraic variety and $X_{\mathbb{C}} \subset \mathbb{C}^n$ its complexification. Fix $u \in \mathbb{R}^n$. Then the \textit {Euclidean distance degree (EDD)} of $X$ is the number of complex critical points of the squared distance function $d_u(x)= \sum_{i=1}^n (u_i-x_i)^2$ on the smooth locus of $X_{\mathbb{C}}$ \cite{EDD}. It can be shown that this number is constant on a dense open subset of data $u \in \mathbb{R}^n$. In \cite{EDDTransfer}, Bik and Draisma prove that the variety of $\lambda$-partitioned eigenvalues and its diagonal restriction have the same EDD. The diagonal restriction is a subspace arrangement, so its EDD is its number of distinct maximal subspaces. 

\begin{theorem}[Bik and Draisma]
Let $\lambda=(\lambda_1, \dots, \lambda_m)$.
The Euclidean distance degree of the variety $V_{\mathbb{R}}(\lambda)$ of $\lambda$-partitioned eigenvalues is 
$\frac{n!}{\lambda_1!\cdots \lambda_m!}$. 
\end{theorem}

\section*{Acknowledgements}

We thank Yulia Alexandr, Juliette Bruce, M\'{a}ty\'{a}s Domokos, Fulvio Gesmundo, Mateusz Micha{\l}ek and Bernd Sturmfels for helpful discussions. We also thank the anonymous referee for many helpful suggestions that improved the paper. This material is based upon work supported by the National Science Foundation Graduate Research Fellowship Program under Grant No.~DGE 1752814. Any opinions, findings, and
conclusions or recommendations expressed in this material are those of the authors and do not necessarily reflect the views of the National Science Foundation.


\newpage

 \begin{thebibliography}{10}

\begin{small}
\setlength{\itemsep}{-0.4mm}



\bibitem{AC}
D.~Andrica and O.L.~Chender: {\em Rodrigues formula for the Cayley transform of groups SO(n) and SE(n)}, Studia Universitatis Babes-Bolyai {\bf 60} (2015) 31--38.

\bibitem{EDDTransfer}
A.~Bik and J.~Draisma: {\em A note on ED degrees of group-stable subvarieties in polar representations}, Israel Journal of Mathematics {\bf 228} (2018) 353--377.

\bibitem{BKL}
P.~Breiding, K.~Kozhasov, A.~Lerario:  {\em On the geometry of the set of symmetric matrices with repeated eigenvalues}, Arnold Mathematical Journal {\bf 4} (2018) 423--443. 

\bibitem{BRW}
P.~Breiding, K.~Ranestad, M.~Weinstein: {\em Enumerative geometry of curvature of algebraic hypersurfaces}. In preparation. 

\bibitem{Derksen}
H.~Derksen: {\em Hilbert series of subspace arrangements}, Journal of Pure and Applied Algebra {\bf 209} (2007) 91--98.

\bibitem{Domokos1}
M.~Domokos: {\em Discriminant of symmetric matrices as a sum of squares and the orthogonal group}, Communications of Pure and Applied Mathematics {\bf 64} (2011) 443--465. 

\bibitem{Domokos3}
M.~Domokos: {\em Hermitian matrices with a bounded number of eigenvalues}, Linear Algebra and its Applications {\bf 12} (2013) 3964--3979.  

\bibitem{Domokos2}
M.~Domokos: {\em Invariant theoretic characterization of subdiscriminants of matrices}, Linear and Multilinear Algebra {\bf 62} (2014) 63--72. 

\bibitem{EDD}
J.~Draisma, E.~Horobe\c{t}, G.~Ottaviani, B.~Sturmfels, R.R.~Thomas: {\em The Euclidean distance degree of an algebraic variety}, Foundations of Computational Mathematics {\bf 16} (2016) 99--149. 

\bibitem{EDDOrtho}
D.~Drusvyatskiy, H.~Lee, G.~Ottaviani, R.R.~Thomas: 
{\em The Euclidean distance degree of orthogonally invariant matrix varieties}, Israel Journal of Mathematics {\bf 221} (2017) 291--316. 

\bibitem{FH}
W.~Fulton and J. Harris: {\em Representation Theory: A First Course}, Graduate Texts in Mathematics, Springer New York, 1991. 

\bibitem{M2}
D.~Grayson and M.~Stillman:
{\em  Macaulay2, a software system for research in algebraic geometry},
available at \url{http://www.math.uiuc.edu/Macaulay2/}.

\bibitem{Humphreys} 
J.E.~Humphreys: {\em Introduction to Lie Algebras and Representation Theory}, Graduate Texts in Mathematics, Springer New York, 2012. 

\bibitem{Ily}
N.V.~Ilyushechkin: {\em The discriminant of the characteristic polynomial of a normal matrix}, Mat. Zametki {\bf 51} (1992), no. 3, 16--23; translation in Math. Notes {\bf 51}(1992), nos. 3--4, 230--235. 

\bibitem{Khaz}
K. Kozhasov: { \em On minimality of determinantal varieties}, {\tt arXiv:2003.01049[math.AG]}. Preprint (2020). 

\bibitem{Lax}
P.D.~Lax: {\em On the discriminant of real symmetric matrices}, Communications of Pure and Applied Mathematics {\bf 51} (1998) 1387--1396. 

\bibitem{MSbook}  M.~Micha{\l}ek and B.~Sturmfels: {\em Invitation to Nonlinear Algebra},
Graduate Studies in Mathematics, American Mathematical Society, 2021.

\bibitem{Parlett}
B.~Parlett: {\em The (matrix) discriminant as a determinant}, Linear Algebra and its Applications { \bf 355} (2002) 85--101. 

\bibitem{RoySOS}
M.F.~Roy: Subdiscriminants of symmetric matrices are sums of squares, {\em Mathematics, Algorithms, Proofs}, Volume 05021 of Dagstuhl Seminar Proceedings, Internationales Begegnungs und Forschungszentrum für Informatik (IBFI), Schloss Dagstuhl, Germany, 2005.

\bibitem{BerndPoly} B.~Sturmfels: {\em Solving Systems of Polynomial Equations}, Conference Board of the Mathematical Sciences, 2002.

\bibitem{TauvelYu}
P.~Tauvel and R.W.T.~Yu: {\em Lie Algebras and Algebraic Groups}, Springer Monographs in Mathematics, Springer Verlag, 2005. 


\end{small}
\end{thebibliography}
\end{document}